\documentclass{amsart}
\usepackage{amssymb,amsmath,latexsym,amsfonts,amsthm,mathrsfs,lineno}
\usepackage{enumerate}
\usepackage{graphicx}

\newtheorem{thm}{Theorem}
\newtheorem*{thm*}{Theorem}

\newtheorem{claimm}{Claim}

\newtheorem{lemma}{Lemma}

\newcommand{\N}{\mathbb{N}}

\newcommand{\Q}{\mathbb{Q}}
\newcommand{\Z}{\mathbb{Z}}
\newcommand{\C}{\mathcal{C}}
\newcommand{\Cy}{\textbf{S}(2)}
\newcommand{\Ty}{\overrightarrow{\mathbb{T}}}
\newcommand{\Rt}{\overrightarrow{\mathcal{R}}}
\newcommand{\T}{\mathbb{T}}
\newcommand{\p}{\mathcal{P}}
\newcommand{\pr}{p}
\newcommand{\pj}{q}
\newcommand{\funct}[2]{#1 \longrightarrow #2}

\newcommand{\om}[1]{\textbf{#1},<^{\textbf{#1}}}
\newcommand{\m}[1]{\textbf{#1}}

\newcommand{\mc}[1]{\widetilde{\textbf{#1}}}

\newcommand{\restrict}[2]{#1 \upharpoonright #2}
\newcommand{\arrows}[3]{\longrightarrow {#1}^{#2}_{#3}}

\newcommand{\Aut}{\mathrm{Aut}}
\newcommand{\la}[1]{\stackrel{#1}{\longleftarrow}}
\newcommand{\ra}[1]{\stackrel{#1}{\longrightarrow}}

\newcommand{\drawat}[3]{\makebox[0pt][l]{\raisebox{#2}{\hspace*{#1}#3}}}

\author{C. Laflamme, L. Nguyen Van Th\'e, N. W. Sauer}

\address{Department of Mathematics and Statistics, University of Calgary, 2500 University Drive NW, Calgary, Alberta, Canada, T2N1N4.}

\email{laf@math.ucalgary.ca}
\email{nguyen@math.ucalgary.ca}
\email{nsauer@math.ucalgary.ca}

\title{Partition properties of the dense local order and a colored version of Milliken's theorem}


\date{September, 2007}

\begin{document}

\begin{abstract}
We study finite dimensional partition properties of the countable homogeneous dense local order (a directed graph closely related to the order structure of the rationals). Some of our results use ideas borrowed from the partition calculus of the rationals and are obtained thanks to a strengthening of Milliken's theorem on trees. 
\end{abstract}

\maketitle

\section{Introduction}

The purpose of this paper is the study of the partition properties of a particular oriented graph, called the \emph{dense local order}. To our knowledge, the dense local order (denoted $\Cy$ in the sequel) appeared first in a work of Woodrow \cite{W}. The attempt then was to characterize the countable tournaments which are homogeneous, that is for which any isomorphism between finite subtournaments can be extended to an automorphism of the whole structure. It was shown that up to isomorphism, there are only two countable homogeneous tournaments which do not embed the tournament $\m{D}$ shown in Figure 1. Those are 1) the tournament corresponding to the rationals $(\Q,<)$ where $x \la{\Q} y$ iff $x<y$ and 2) the dense local order $\Cy$.

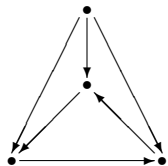
\begin{figure}[h]
\setlength{\unitlength}{1mm}
\begin{picture}(20,30)(0,0)

\put(10,15){\circle*{1}}
\put(10,25){\circle*{1}}
\put(0,5){\circle*{1}}
\put(20,5){\circle*{1}}
\put(1,5){\vector(1,0){18}}
\put(9,14){\vector(-1,-1){8}}
\put(19,6){\vector(-1,1){8}}
\put(10,24){\vector(0,-1){8}}
\put(9,24){\vector(-1,-2){9}}
\put(11,24){\vector(1,-2){9}}

\end{picture}
\caption{The tournament $\m{D}$}
\end{figure}

The tournament $\Cy$ is defined as follows: let $\mathbb{T}$ denote the unit circle in the complex plane. Define an oriented graph structure on $\mathbb{T}$ by declaring that there is an arc from $x$ to $y$ iff $0 < \arg (y/x) < \pi$. Call $\Ty$ the resulting oriented graph. The dense local order is then the substructure $\Cy$ of $\Ty$ whose vertices are those points of $\T$ with rational argument. 

A few years later, Lachlan proved in \cite{L} that any countable homogeneous tournament embedding $\m{D}$ also embeds every finite tournament. This completed the classification initiated by Woodrow and showed that up to isomorphism there are only three countable homogeneous tournaments: the rationals, the dense local order and the countable random tournament $\Rt$ (up to isomorphism, the unique countable homogeneous tournament into which every countable tournament embeds). Note that this is in sharp contrast with the more general case of countable homogeneous oriented graphs as there are continuum many such objects (this latter result is due to Henson \cite{H} while the classification of countable homogeneous graphs is due to Cherlin \cite{Ch}). In this paper, we will be interested in Ramsey type questions with the following flavor: given $k \in \N$ and a finite tournament $\m{Y}$, is there a finite tournament $\m{Z}$ such that for every $k$-coloring of the arcs of $\m{Z}$, there is an induced copy $\mc{Y}$ of $\m{Y}$ in $\m{Z}$ where all the arcs have the same color? For this particular problem, the answer could be negative (depending on which $\m{Y}$ we started with) but becomes positive if one is allowed to have at most two colors instead of one single color for the arcs of $\mc{Y}$. More generally, Ramsey-theoretic properties of the rationals and of the random tournament are known in the following sense: given tournaments $\m{X}$, $\m{Y}$ and $\m{Z}$, we write $\m{X} \subset \m{Z}$ when $\m{X}$ is an induced subtournament of $\m{Z}$ and $\m{X} \cong \m{Y}$ when there is an isomorphism from $\m{X}$ onto $\m{Y}$. We define the set $\binom{\m{Z}}{\m{X}}$ as \[\binom{\m{Z}}{\m{X}} = \{ \mc{X} \subset \m{Z} : \mc{X} \cong \m{X}  \} \enspace .\] 

For $k,l$ positive elements of $\N$ (throughout this article, $\N = \{ 0, 1, 2, 3, \ldots \}$) and a
triple $\m{X}, \m{Y}, \m{Z}$ of tournaments, the symbol \[\m{Z} \arrows{(\m{Y})}{\m{X}}{k,l} \] is an abbreviation for the statement: ``For any $\chi : \funct{\binom{\m{Z}}{\m{X}}}{[k]}$ (by $[k]$ we mean the set $\{0,\ldots,k-1\}$), there is $\widetilde{\m{Y}} \in \binom{\m{Z}}{\m{Y}}$ such that $\chi$ does not take more than $l$ values on $\binom{\widetilde{\m{Y}}}{\m{X}}$.'' When $l = 1$, this is simply
written $\m{Z} \arrows{(\m{Y})}{\m{X}}{k}$. Let $\mathcal{Q}$, $\mathcal{T}$ and $\mathcal{C}$ denote the class of all finite subtournaments of $\Q$, $\Rt$ and $\Cy$ respectively. For $\mathcal{K} = \mathcal{Q}, \mathcal{T}$ or $\mathcal{C}$ and $\m{X} \in \mathcal{K}$, a first problem is to determine the value of the \emph{Ramsey degree of $\m{X}$ in $\mathcal{K}$}, denoted $t_{\mathcal{K}}(\m{X})$, defined as the least $l \in \N \cup \{ \infty\}$ such that for every $\m{Y} \in \mathcal{K}$ and every $k \in \N$ there exists $\m{Z} \in \mathcal{K}$ such that \[\m{Z} \arrows{(\m{Y})}{\m{X}}{k,l} \enspace . \] 

A second problem is to determine the value of the \emph{big Ramsey degree of $\m{X}$ in $\mathcal{K}$}. This latter quantity is denoted $T_{\mathcal{K}}(\m{X})$ and is defined as follows: let $\m{F}$ denote the tournament $\Q$ if $\mathcal{K} = \mathcal{Q}$, $\Rt$ if $\mathcal{K} = \mathcal{T}$ and $\Cy$ if $\mathcal{K} = \C$. Then the big Ramsey degree of $\m{X}$ in $\mathcal{K}$ is the least $L \in \N \cup \{ \infty\}$ such that for every $k \in \N$, \[\m{F} \arrows{(\m{F})}{\m{X}}{k,L} \enspace . \] 

For $\mathcal{K} = \mathcal{Q}$, the Ramsey degrees and the big Ramsey degrees are always finite and can be computed effectively. More precisely, every $\m{X} \in \mathcal{Q}$ is such that $t_{\mathcal{Q}}(\m{X}) = 1$. This is an easy consequence of the original Ramsey theorem. By contrast, a much more difficult proof due to Devlin in \cite{D} showed that $T_{\mathcal{Q}}(\m{X}) = \tan^{(2|\m{X}|-1)}(0)$, the $(2|\m{X}|-1)$st derivative of $\tan$ evaluated at $0$. Recall that $\tan'(0)=1$, $\tan^{(3)}(0)=2$, $\tan^{(5)}(0)=16$, $\tan^{(7)}(0)=272$ and that in general \[ \tan^{(2n-1)}(0)=\frac{B_{2n}(-4)^n (1-4^n)}{2n}\] where $(B_n)_{n \in \N}$ is the Bernouilli sequence defined (for example) by \[ \frac{x}{e^x -1}=\sum_{n=0}^{+\infty}\frac{B_n}{n!}x^n \ \ \textrm{for every $|x|<2\pi$} \enspace .\]

For $\mathcal{K} = \mathcal{T}$, Ramsey degrees and big Ramsey degrees have never been studied explicitly but can be determined thanks to other known Ramsey type results. In particular, thanks to a general partition result of Ne\v{s}et\v{r}il and R\"odl, it is known that every $\m{X} \in \mathcal{T}$ has a finite Ramsey degree and that \[t_{\mathcal{T}}(\m{X}) = |\m{X}|!/|\Aut(\m{X})|\] where $\mathrm{Aut}(\m{X})$ denotes the set of all automorphisms of $\m{X}$. On the other hand, $T_{\mathcal{T}}(\m{X})$ is known to be finite and the work \cite{LSV} by Laflamme, Sauer and Vuksanovic on the countable random undirected graph actually shows that its value can be interpreted as the number of representations that $\m{X}$ admits into a certain well-known finite structure (that is, there is an algorithm for every $\m{X}$ determining the value of $T_{\mathcal{T}}(\m{X})$). However, it is still unclear whether this expression can be simplified so as to give a counterpart to Devlin's formula in the context of $\mathcal{T}$.

As for the case $\mathcal{K} = \C$, it does not seem to have been studied by anybody so far and the purpose of the present paper is therefore to fill that gap. We first study the Ramsey degrees in $\C$. Our result here reads as follows: 

\begin{thm}
\label{thm:RdC}
Every element $\m{X}$ of $\C$ has a
Ramsey degree in $\C$ equal to \[t_{\C}(\m{X}) = 2|\m{X}|/|\Aut(\m{X})| \enspace .\]
\end{thm}

We then turn to the study of the big Ramsey degrees in $\C$, and prove:

\begin{thm}
\label{thm:bigRdC}
Every element $\m{X}$ of $\C$ has a big Ramsey degree in $\C$ equal to \[T_{\C}(\m{X}) = t_{\C}(\m{X})\tan ^{(2|\m{X}|-1)}(0) \enspace . \]
\end{thm}

As a direct corollary, for every natural $k>0$ and every coloring $\chi : \funct{\Cy}{[k]}$, there is an isomorphic copy of $\Cy$ inside $\Cy$ on which $\chi$ takes only $2$ colors (this statement is not as obvious as it looks), and $2$ is the best possible bound. On the other hand, for every $k$-coloring of the arcs of $\Cy$, there is an isomorphic copy of $\Cy$ inside $\Cy$ where only $8$ colors appear, and $8$ is the best possible bound.

Theorem \ref{thm:RdC} and Theorem \ref{thm:bigRdC} are proved thanks to a connection between the class $\C$ and some other classes of finite structures for which several Ramsey properties are already known. Those are the classes $\p _n$ of all finite structures of the form $\m{A} = (A, <^{\m{A}}, P_1 ^{\m{A}},\ldots,P_n ^{\m{A}})$ where $<^{\m{A}}$ is a linear ordering on $A$ and $\{P_1 ^{\m{A}},\ldots,P_n ^{\m{A}} \}$ is a partition of $A$ into disjoint sets. Given two such structures $\m{A}$ and $\m{B}$, an isomorphism is an order-preserving bijection $f$ from $A$ to $B$ such that for every $x \in A$, $x \in P_i^{\m{A}}$ iff $f(x) \in P_i^{\m{B}}$. As it was the case for the class $\C$, there is a unique countable homogeneous structure whose class of finite substructures is $\p _n$. In this paper, this structure is denoted $\Q _n$. The role that $\Q_n$ plays with respect to $\p _n$ is exactly the same as the role that $\Cy$ plays for the class $\C$. As for $\Cy$, the structure $\Q_n$ can be represented quite simply. Namely, the structure $\Q _n$ can be seen as $(\Q , Q_1 ,\ldots,Q_n,<)$ where $\Q$ denotes the rationals, $<$ denotes the usual ordering on $\Q$, and every $Q_i$ is a dense subset of $\Q$. The notions of Ramsey degrees and big Ramsey degrees in $\p _n$ are then defined in exactly the same way as they are for $\C$. The Ramsey degrees in $\p _n$ are known: every element in $\p_n$ has a Ramsey degree in $\p _n$ equal to one. This result, in the case $n=2$, is one of the key facts in our proof of Theorem \ref{thm:RdC}. As for the big Ramsey degrees, we are able to prove that:

\begin{thm}
\label{thm:bigRdP}
Let $n$ be a positive natural. Then every element $\m{X}$ of $\p_n$ has a big Ramsey degree in $\p _n$ equal to $\tan ^{(2|\m{X}|-1)}(0)$.
\end{thm}   

Equivalently, for every element $\m{X}$ of $\p _n$, $\tan ^{(2|\m{X}|-1)}(0)$ is the least possible natural such that for every natural $k>0$, \[ \Q _n \arrows{(\Q _n)}{\m{X}}{k, \tan ^{(2|\m{X}|-1)}(0)}\enspace .\]

Again, the corresponding result for $n=2$ turns out to be crucial for our purposes. Here, it is one of the ingredients of our proof of Theorem \ref{thm:bigRdC}. Theorem \ref{thm:bigRdP} is obtained by following ideas borrowed from Devlin \cite{D} together with a strengthening of a theorem of Milliken \cite{Mi}: consider a finitely branching tree (in the order-theoretic sense) $T$ of infinite height, a number $m$, and a subset $S\subset T$. If $S$ satisfies certain properties listed in Section \ref{section:Milliken}, we say that $S$ is a \emph{strong subtree of $T$ of height $m$}. According to Milliken's theorem, if we assign a color to each strong subtree of height $m$ out of a finite family of colors then there exists a strong subtree of infinite height such that all strong subtrees of height $m$ contained in it have the same color. In the version we need in order to prove Theorem \ref{thm:bigRdP}, each level of the tree is assigned a color (out of a finite set not related to the set of colors of subtrees). We then consider only strong subtrees of height $m$ with some given level-coloring structure and we look
for a strong subtree of infinite height with a level-coloring structure similar to that of the original tree.

The paper is organized as follows: in section \ref{section:Extensions}, we define the notion of extension in $\p _2$ for any element of $\C$ and show that the number of nonisomorphic extensions in $\p_2$ of a given element of $\C$ can be expressed simply in terms of the size of its automorphism group. In section \ref{section:RdC}, we use this result to compute Ramsey degrees in $\C$ and to prove Theorem \ref{thm:RdC}. In section \ref{section:bigRdC} we turn to the study of big Ramsey degrees and show how Theorem \ref{thm:bigRdC} follows from Theorem \ref{thm:bigRdP}. The two remaining sections of the paper are devoted to a proof of Theorem \ref{thm:bigRdP}. The first step is carried out in section \ref{section:Milliken} where we prove a strengthening of Milliken's theorem on trees. Together with Devlin's original ideas from \cite{D}, this result is then used to derive Theorem \ref{thm:bigRdP}. 

\

\textbf{Ackowledgements}: C. Laflamme was supported by NSERC of Canada Grant\# 690404. L. Nguyen Van Th\'e would like to thank the support 
of the Department of Mathematics \& Statistics Postdoctoral Program at the University of Calgary. N. W.  Sauer was supported by NSERC of Canada Grant \# 691325. We would also like to thank the anonymous referee whose numerous and helpful comments improved the paper considerably.  

\section{Extensions of circular tournaments}

\label{section:Extensions}

The purpose of this section is to establish a connection between the elements of $\C$ and the elements of $\p _2$. This connection is not new: it already appears in \cite{L} and in \cite{Ch} as well as in several other papers. Here, it enables us to deduce most of our results from an analysis of the partition calculus on $\p _2$. This is done by defining a notion of \emph{extension} for every element of $\C$: 

For $\m{A} = (A; <^{\m{A}}, P_1 ^{\m{A}}, P_2 ^{\m{A}})$ with $<^{\m{A}}$ a linear ordering on $A$ and $(P_1 ^{\m{A}}, P_2 ^{\m{A}})$ a partition of $A$, let $\sim ^{\m{A}}$ denote the equivalence relation induced by $(P_1 ^{\m{A}}, P_2 ^{\m{A}})$. Let then $\pr (\m{A})$ denote the oriented graph based on $A$ and equipped with the arc relation denoted $\la{\m{A}}$ and such that \[a \la{\m{A}} b \ \ \textrm{iff} \ \  \left((a \sim ^{\m{A}} b \ \ \textrm{and} \ \ a<^{\m{A}}b) \ \ \textrm{or} \ \ (a \nsim ^{\m{A}} b \ \ \textrm{and} \ \ b<^{\m{A}}a)\right)\enspace .\] 

This construction is illustrated in Figure \ref{fig:construction}. 

\begin{figure}[h]
\begin{center}
\hskip-10pt\includegraphics{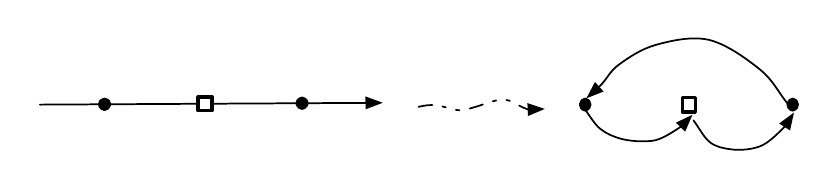}%
\drawat{-75.7mm}{10.5mm}{$a_1$}%
\drawat{-65mm}{10.5mm}{$a_2$}%
\drawat{-56mm}{10.5mm}{$a_3$}%
\drawat{-28mm}{10.5mm}{$a_1$}%
\drawat{-17mm}{10.5mm}{$a_2$}%
\drawat{-5.4mm}{10.5mm}{$a_3$}%
\drawat{-78.88mm}{0.68mm}{$\m{A}$}%
\drawat{-29mm}{0.68mm}{$\pr(\m{A})$}%
\end{center}
\caption{Construction of $p(A)$}
\label{fig:construction}
\end{figure}

In words: interpret $<^{\m{A}}$ as a directed graph relation $\la{}$ where $x \la{} y$ iff $x <^{\m{A}} y$. Then reverse all the arcs between the elements of $\m{A}$ which are not $\sim ^{\m{A}}$-equivalent. It should be clear that $\pr (\m{A})$ is a tournament. For a tournament $\m{X}$, any $\m{A}$ such that $\pr (\m{A}) =  \m{X}$ is called an \emph{extension of $\m{X}$}.

\begin{lemma}

\label{lem:ext}

Let $\m{A} \in \p _2$. Then $\pr(\m{A}) \in \C$.  

\end{lemma}

\begin{proof} We construct $\varphi(\m{A}) \subset \Cy$ isomorphic to $\pr(\m{A})$ as follows: denote by $Im^+$ the complex open upper half plane. The directed graph structure on $\Cy$ induces a linear ordering on $\Cy \cap Im^+$ if we set $x<y$ iff $x \la{\Cy} y$. As a linear order, $(\Cy \cap Im^+, <)$ is isomorphic to $\Q$, hence without loss of generality we may assume that the linear ordering $(\om{A})$ is a subset of $(\Cy \cap Im^+, <)$. Using the fact that in the complex plane, $(-a)$ is the symmetric of $a$ with respect to the origin, let $\varphi : \funct{\m{A}}{\Cy}$ be defined by: \begin{displaymath}
\varphi (a) = \left \{ \begin{array}{cl}
 a & \textrm{if $a \in P_1 ^{\m{A}}$,} \\
 -a & \textrm{if $a \in P_2 ^{\m{A}}$.}
 \end{array} \right.
\end{displaymath} 

Observe that if $a, a' \in \m{A}$ belong to the same $P_i ^{\m{A}}$, then $\varphi$ preserves the arc relation $\la{\Cy}$ between $a$ and $a'$ while it reverses it when $a$ and $a'$ do not belong to the same $P_i ^{\m{A}}$. This fact together with the construction scheme described previously for $\pr(\m{A})$ (paragraph preceding Lemma \ref{lem:ext}) imply that the tournaments $\varphi(\m{A})$ and $\pr(\m{A})$ are isomorphic. \end{proof}

The procedure applied in Lemma \ref{lem:ext} (refered to as \emph{projection procedure} in the sequel) is illustrated in a simple case in Figure \ref{fig:projection}. 

\begin{figure}[h]
\begin{center}
\hskip-10pt\includegraphics{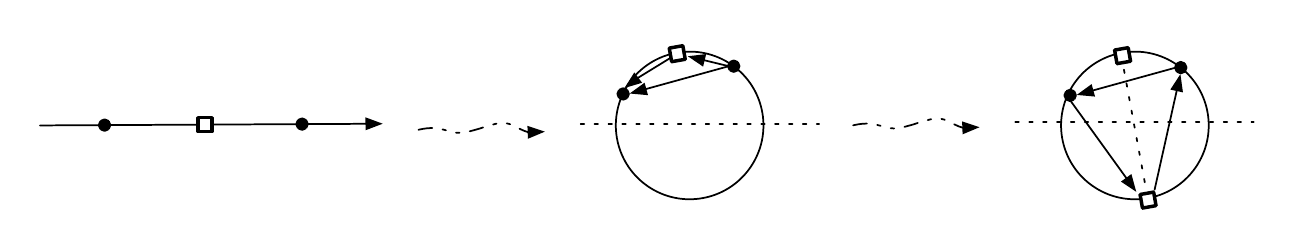}%
\drawat{-122mm}{14.64mm}{$a_1$}%
\drawat{-111.65mm}{14.64mm}{$a_2$}%
\drawat{-99.08mm}{14.64mm}{$a_3$}%
\drawat{-123.88mm}{4.68mm}{$\m{A}$}%
\drawat{-73.2mm}{4.68mm}{$\m{S}(2)$}%
\drawat{-28.75mm}{4.68mm}{$\m{S}(2)$}%
\drawat{-71.67mm}{17mm}{$a_1$}%
\drawat{-63.9mm}{22mm}{$a_2$}%
\drawat{-57.09mm}{20mm}{$a_3$}%
\drawat{-29.75mm}{17mm}{$\varphi(a_1)$}%
\drawat{-20.27mm}{2mm}{$\varphi(a_2)$}%
\drawat{-11.23mm}{19.75mm}{$\varphi(a_3)$}%
\drawat{-51.39mm}{15.75mm}{$Im^+$}%
\end{center}
\caption{The projection procedure --- construction of $\varphi(A)$}
\label{fig:projection}
\end{figure}

\begin{lemma}
\label{lem:pi2}
Let $\m{X} \subset \Cy$. Then $\m{X}$ has exactly $2|\m{X}|/|\Aut(\m{X})|$ nonisomorphic extensions.
\end{lemma}

\begin{proof}
We first show that the projection procedure to obtain $\pr(\m{A})$ from $\m{A}$ can be reversed to an \emph{extension procedure} in order to construct extensions of $\m{X}$: Consider a line $L$ through the origin avoiding all the vertices of $\m{X}$. Choose one of the open half planes with boundary $L$, call it $H$. Then, set: \[ P_1 ^{\m{A}} = X \cap H, \ \ P_2 ^{\m{A}} = \{ -a : a \in X \smallsetminus H \}\enspace .\] 

That is, $P_2 ^{\m{A}} $ is the set obtained from $X \smallsetminus H $ by symmetry with respect to the origin. As previously, the arc relation on $\Cy$ induces a linear ordering on $A:= P_1 ^{\m{A}} \cup P_2 ^{\m{A}}$, call it $<^{\m{A}}$. Then the structure $\m{A}:=(A; P_1 ^{\m{A}}, P_2 ^{\m{A}}, <^{\m{A}})$ is in $\p _2$ and is an extension of $\m{X}$. A simple application of the extension procedure is illustrated in Figure \ref{fig:extension}. 

\begin{figure}[h]
\begin{center}
\hskip-10pt\includegraphics{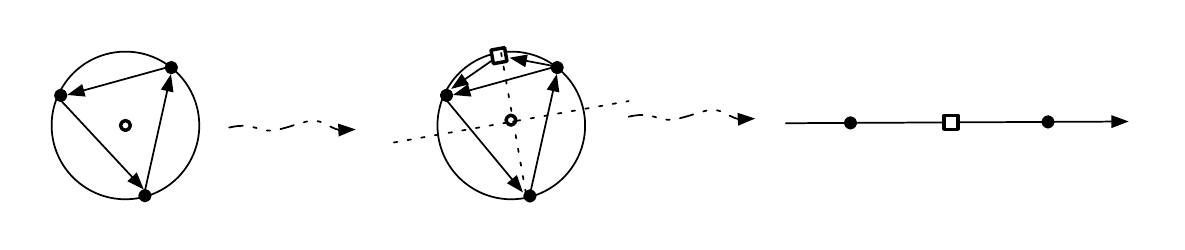}%
\drawat{-116.71mm}{16.52mm}{$x$}%
\drawat{-105.7mm}{3mm}{$y$}%
\drawat{-103.13mm}{20mm}{$z$}%
\drawat{-77.88mm}{16.52mm}{$x$}%
\drawat{-67.5mm}{3mm}{$y$}%
\drawat{-72.54mm}{22mm}{$-y$}%
\drawat{-63.2mm}{19.5mm}{$z$}%
\drawat{-35.3mm}{15.8mm}{$x$}%
\drawat{-27.58mm}{15.8mm}{$-y$}%
\drawat{-15.64mm}{15.8mm}{$z$}%
\drawat{-59.27mm}{18mm}{$H$}%
\drawat{-55.23mm}{15mm}{$L$}%
\end{center}
\caption{The extension procedure}
\label{fig:extension}
\end{figure}

Note the following essential fact: If $\m{A}$ is an extension of $\m{X}$ then applying the projection procedure to $\m{A}$ produces a copy $\mc{X}$ of $\m{X}$ included in $\Cy$, and applying the extension procedure to this same $\mc{X}$ where $L$ is the real axis and $H$ is the open upper half plane produces $\m{A}$ itself. It follows that every extension of $\m{X}$ can be obtained by applying the extension procedure to $\m{X}$. 

Hence, to count the number of non isomorphic extensions of $\m{X}$ in $\p _2$, we need to know when different choices of $L, H$ provide non isomorphic extensions. Observe first that the choice of $L$ determines a linear ordering on $\m{X}$ as follows: Choose any of the two half planes with boundary $L$. Using symmetry with respect to the complex origin if necessary, bring all the points of $\m{X}$ inside this half plane, where the arc relation on $\Cy$ induces a linear ordering. Then, simply pull this linear ordering back to $\m{X}$. Note that the linear ordering we obtain on $\m{X}$ does not depend on the half plane we chose to construct it. Observe that if two lines $L, L'$ induce linear orderings $<, <'$ such that $(\m{X}, <)$ and $(\m{X}, <')$ are non isomorphic (when seen as ordered tournaments), then any choice of $H, H'$ leads to non isomorphic extensions of $\m{X}$ in $\p_2$. Since for each line $L$ there are two choices for $H$, it follows that the number of non isomorphic extensions of $\m{X}$ in $\p_2$ is twice the number of structures of the form $(\m{X}, <)$ where $<$ comes from a line. 

To compute this number, observe that two lines $L, L'$ induce the same linear ordering on $\m{X}$ when their half planes contain the same vertices of $\m{X}$. Therefore, there are $|\m{X}|$ such orderings. Next, consider $<$ and $<'$. They enumerate $\m{X}$ increasingly as $\{ x_1,\ldots , x_{|\m{X}|}\}$ and $\{ x'_1,\ldots , x'_{|\m{X}|}\}$ respectively, and $(\m{X}, <)$ and $(\m{X}, <')$ are isomorphic exactly when the map $x_n \mapsto x'_n$ is an automorphism of $\m{X}$. Therefore, there are essentially $|\m{X}|/|\Aut(\m{X})|$ different ways to order $\m{X}$ via a line. The result of Lemma \ref{lem:pi2} follows. \end{proof}

\textbf{Remark:} Observe that since the number $|\m{X}|/|\Aut(\m{X})|$ represents the number of different ways to order $\m{X}$ via a line, it is an integer. Therefore, $|\Aut(\m{X})|$ divides $|\m{X}|$. 

\section{Ramsey degrees in $\C$}

\label{section:RdC} 

For $\m{X} \in \C$, we write $t (\m{X})$ for the number $2|\m{X}|/|\Aut(\m{X})|$. The purpose of this section is to prove Theorem \ref{thm:RdC}, that is: Every $\m{X} \in \C$ has a finite Ramsey degree $t_{\C}(\m{X})$ in $\C$ and $ t_{\C} (\m{X}) = t (\m{X})$. Throughout this section, $\m{X} \in \C$ is fixed. We first show that $ t_{\C} (\m{X}) \leq t (\m{X})$ and next that $t(\m{X}) \leq t_{\C}(\m{X})$.

\subsection{Upper bound for $ t_{\C} (\m{X})$: $t_{\C}(\m{X}) \leq t(\m{X})$}

We need to prove that for every strictly positive $k \in \N$, every $\m{Y} \in \C$, there is $\m{Z} \in \C$ such that \[ \m{Z} \arrows{(\m{Y})}{\m{X}}{k,t(\m{X})}\enspace .\] 

This is done thanks to the following partition property for $\p _2$: 

\begin{thm}[Kechris-Pestov-Todorcevic, \cite{KPT}]
\label{thm:RPp2}
Let $n \in \N$, $\m{A}, \m{B} \in \p _n$ and $k$ a positive natural. Then there is $\m{C} \in \p _n$ such that \[ \m{C} \arrows{(\m{B})}{\m{A}}{k}\enspace .\]
\end{thm}

\begin{proof}
cf \cite{KPT}, Theorem 8.4, p.158-159. 
\end{proof}

In order to prove that $\m{X}$ has a finite Ramsey degree $t_{\C}(\m{X})$ and that $t_{\C}(\m{X}) \leq t(\m{X})$, we apply Theorem \ref{thm:RPp2} $t (\m{X})$ times as follows. For the sake of clarity, we only consider the particular case where $t (\m{X}) = 2$ but it should be clear at the end of the argument how to generalize to any other value. According to Lemma \ref{lem:pi2}, $t (\m{X})$ is equal to the number of nonisomorphic extensions of $\m{X}$ in $\p _2$. Let $\m{A}_0, \m{A}_1$ denote those extensions. Let also $\m{B}_0 \in \p _2$ be such that $\pr (\m{B}_0) \cong \m{Y}$. Using Theorem \ref{thm:RPp2}, construct $\m{B}_1$ so that \[ \m{B}_{1} \arrows{(\m{B}_0)}{\m{A}_0}{k}\enspace .\] Next, construct $\m{B}_2$ so that \[ \m{B}_{2} \arrows{(\m{B}_1)}{\m{A}_1}{k}\enspace .\]

We claim that $\m{Z} := \pr (\m{B}_{2})$ is as required. Let $\chi : \funct{\binom{\m{Z}}{\m{X}}}{[k]}$. Then $\chi$ induces a map from $\binom{\m{B}_{2}}{\m{A}_1}$ to $[k]$. By construction of $\m{B}_2$, we can find $\mc{B}_1 \in \binom{\m{B}_2}{\m{B}_1}$ such that $\chi$ is constant on $\binom{\mc{B}_1}{\m{A}_1}$. Then, working in $\mc{B}_1$, $\chi$ induces a map from $\binom{\mc{B}_1}{\m{A}_0}$ to $[k]$. By construction of $\m{B}_1$, we can find $\mc{B}_0 \in \binom{\mc{B}_1}{\m{B}_0}$ such that $\chi$ is constant on $\binom{\mc{B}_0}{\m{A}_0}$. Note that since $\mc{B}_0 \subset \mc{B}_1$, $\chi$ is also constant on $\binom{\mc{B}_0}{\m{A}_1}$. In $\m{Z}$, the substructure $\mc{Y}$ supported by $\mc{B}_{0}$ is then isomorphic to $\m{Y}$ and we have \[\binom{\mc{Y}}{\m{X}} = \binom{\mc{Y}}{\m{A}_0} \cup \binom{\mc{Y}}{\m{A}_1} \enspace .\] 

Therefore, the map $\chi$ takes no more than $2$ values (in the general case, $t(\m{X})$ values) on $\binom{\mc{Y}}{\m{X}}$, as required. Thus, $\m{X}$ has a Ramsey degree $t_{\C}(\m{X})$ in $\C$ and \[ t_{\C}(\m{X}) \leq t(\m{X})\enspace .\] 

\subsection{Lower bound for $t_{\C}(\m{X})$: $t(\m{X}) \leq t_{\C}(\m{X})$}

The main ingredient is the following lemma: 

\begin{lemma}
\label{lem:extprop}
There exists $\m{Y} \in \C$ such that every extension of $\m{X}$ embeds into every extension of $\m{Y}$.
\end{lemma}

\begin{proof}
Let $\m{C}_n$ denote the subtournament of $\Cy$ whose set of vertices is given by $\{ e^{\frac{2ik\pi}{2n+1}}: k =0,\ldots , 2n\}$. Observe that up to an interchange of the parts, all the extensions of $\m{C}_n$ in $\p _2$ are isomorphic. Essentially, this is so because there is only one way to order $\m{C}_n$ via a line through the origin as in Lemma \ref{lem:pi2}. Another way to see it is to notice that $\m{C}_n$ admits exactly $2n+1$ automorphisms: every rotation whose angle is a multiple of $(2i\pi/2n+1)$ provides an automorphism. Furthermore, we saw with the Remark at the end of section \ref{section:Extensions} that the number of automorphisms divides the cardinality of the structure. Thus, there cannot be more than $2n+1$ automorphisms, which means in the present case that there are exactly $2n+1$ automorphisms. Therefore, $\m{C}_n$ has $2|\m{C}_n|/|\Aut(\m{C}_n)|=2(2n+1)/(2n+1)=2$ extensions in $\p_2$, namely \[\m{D}_n = \bigl( \left[2n+1 \right], <, \left[ 2n+1 \right]\cap 2\Z, \left[ 2n+1 \right] \cap (2\Z+1) \bigr), \] \[ \m{E}_n = \bigl( \left[ 2n+1 \right], <, \left[ 2n+1 \right] \cap (2\Z+1), \left[ 2n+1 \right] \cap 2\Z \bigr)\enspace .\] 

Note that if $n$ is large enough, then $\m{X}$ embeds into $\m{C}_n$. Note also that seeing $\m{X}$ as a subtournament of $\m{C}_n$, the extension procedure applied to $\m{X}$ with any line $L$ and plane $H$ also induces an extension of $\m{C}_n$. It follows that \emph{any} extension of $\m{X}$ embeds in $\m{D}_n$ and $\m{E}_n$, and we can take $\m{Y}=\m{C}_n$. \end{proof}

Here is how Lemma \ref{lem:extprop} leads to the required inequality: Let $\m{Z} \in \C$. We show that there is a map $\chi$ on $\binom{\m{Z}}{\m{X}}$ using $t (\m{X})$ values and taking $t(\m{X})$ values on the set $\binom{\mc{Y}}{\m{X}}$ whenever $\mc{Y} \in \binom{\m{Z}}{\m{Y}}$. Let $\m{C}$ be an extension of $\m{Z}$ in $\p _2$. Then given a copy $\mc{X}$ of $\m{X}$ in $\m{Z}$, the substructure of $\m{C}$ supported by $\mc{X}$ is an extension of $\m{X}$ in $\p _2$ and is isomorphic to a unique element $\m{A}_j$ of the family $(\m{A}_i)_{i<t (\m{X})}$. Let $\chi (\mc{X}) = j$. Then the map $\chi$ is as required.

\subsection{Comments about $t_{\C}(\m{X})$} 

The effective computation of $t_{\C}(\m{X})$ (or equivalently of $|\Aut(\m{X})|$) in the general case does not seem to be easy. It can be carried out in the most elementary cases, see Figure \ref{fig:elemRd}. 

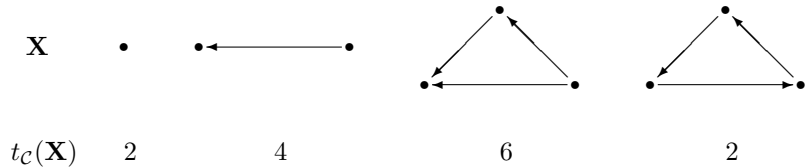
\begin{figure}[h]
\setlength{\unitlength}{1mm}
\begin{picture}(110,30)(0,0)

\put(2,19){$\m{X}$}
\put(0,5){$t_{\C}(\m{X})$}

\put(15,20){\circle*{1}}
\put(15,5){$2$}

\put(25,20){\circle*{1}}
\put(45,20){\circle*{1}}
\put(44,20){\vector(-1,0){18}}
\put(35,5){$4$}

\put(65,25){\circle*{1}}
\put(55,15){\circle*{1}}
\put(75,15){\circle*{1}}
\put(74,15){\vector(-1,0){18}}
\put(64,24){\vector(-1,-1){8}}
\put(74,16){\vector(-1,1){8}}
\put(65,5){$6$}

\put(95,25){\circle*{1}}
\put(85,15){\circle*{1}}
\put(105,15){\circle*{1}}
\put(86,15){\vector(1,0){18}}
\put(94,24){\vector(-1,-1){8}}
\put(104,16){\vector(-1,1){8}}
\put(95,5){$2$}

\end{picture}
\label{fig:elemRd}
\caption{Elementary values of $t_{\C}(\m{X})$}
\end{figure}

There are also a few particular elements of $\C$ for which it can be performed directly. For example, for the oriented graph corresponding to the linear order on $n$ points, the Ramsey degree in $\C$ is equal to $2n$ as there is only one automorphism. On the other hand, call $\m{C}_n$ the subtournament of $\Cy$ whose set of vertices is given by $\{ e^{\frac{2ik\pi}{2n+1}}: k =0,\ldots , 2n\}$. We saw in the proof of Lemma \ref{lem:extprop} that $\m{C}_n$ only has two non isomorphic extensions. It follows that the Ramsey degree of $\m{C}_n$ in $\C$ is equal to $2$. Finally, note that given any $n \in \N$, there are exactly $2^n$ nonisomorphic structures in $\p _2$ whose base set has exactly $n$ elements. Note also that any such structure is the extension of a unique $\m{X} \in \C$ such that $|\m{X}|=n$. It follows that \[ \sum _{|\m{X}|=n} t_{\C}(\m{X}) = 2^n\enspace .\] 

Using the expression of $t_{\C}(\m{X})$, it follows that \[ \sum _{\m{X} \in \C, |\m{X}|=n} \frac{n}{\Aut(\m{X})} = 2^{n-1}\enspace .\] 

\section{Big Ramsey degrees in $\C$}

\label{section:bigRdC}

The purpose of this section is to prove Theorem \ref{thm:bigRdC} under the assumption that Theorem \ref{thm:bigRdP} holds. Denoting by $T(\m{X})$ the number $t_{\C} (\m{X})\tan ^{(2|\m{X}|-1)}(0)$, we need to show that every $\m{X}$ has a finite big Ramsey degree $T_{\C}(\m{X})$ in $\C$ equal to $T(\m{X})$. Equivalently, we first need to prove that for every $k\in \N$, \[ \Cy \arrows{(\Cy)}{\m{X}}{k,T(\m{X})}\enspace .\] 

Then, when this is done, we need show that $T(\m{X})$ is the least number with that property. 

\subsection{Upper bound for $ T_{\C} (\m{X})$: $T_{\C}(\m{X}) \leq T(\m{X})$}

Recall given a structure $\m{A}=(A, <^{\m{A}}, P_1 ^{\m{A}}, P_2 ^{\m{A}})$ where $<^{\m{A}}$ is a linear ordering on $A$ and $(P_1 ^{\m{A}}, P_2 ^{\m{A}})$ is a partition of $A$ into two disjoint sets, the tournament $\pr (\m{A})$ is obtained by interpreting $<^{\m{A}}$ as a directed graph relation $\la{}$ ($x \la{} y$ iff $x <^{\m{A}} y$) and reversing all the arcs between the elements of $\m{A}$ which are in different parts $P_i ^{\m{A}}$. 

\begin{lemma}
\label{claimm:Q2}
$\Q _2$ is an extension of $\Cy$. 
\end{lemma}

\begin{proof}
Applying the extension procedure (described in the proof of Lemma \ref{lem:pi2}) to the tournament $\Cy $ where $L$ is any line through the origin avoiding $\Cy$ and $H$ any of the open half planes with boundary $L$, we get $\Q _2$. Therefore, $\Q _2$ is an extension of $\Cy \smallsetminus \{ 1\} \cong \Cy$. \end{proof}

Keeping the result of Lemma \ref{claimm:Q2} in mind, here is how we prove $\Cy \arrows{(\Cy)}{\m{X}}{k,T(\m{X})}$. Let $\chi : \funct{\binom{\Cy}{\m{X}}}{[k]}$. Let $(\m{A}_i)_{i<t_{\C}(\m{X})}$ enumerate the extensions of $\m{X}$. Then by Lemma \ref{claimm:Q2}, $\chi$ induces a map from $\binom{\Q_2}{\m{A}_0}$ to $[k]$. By Theorem \ref{thm:bigRdP}, we can find $Q_2^0 \in \binom{\Q_2}{\Q_2}$ such that $\chi$ takes no more than $\tan ^{(2|\m{A}_0|-1)}(0) = \tan ^{(2|\m{X}|-1)}(0)$ values on $\binom{Q_2^0}{\m{A}_0}$. Then, working in $Q_2^0$, $\chi$ induces a map from $\binom{Q_2^0}{\m{A}_1}$ to $[k]$. Again, applying Theorem \ref{thm:bigRdP}, we can find $Q_2^1 \in \binom{Q_2^0}{\Q_2}$ such that $\chi$ takes no more than $\tan ^{(2|\m{A}_1|-1)}(0) = \tan ^{(2|\m{X}|-1)}(0)$ values on $\binom{Q_2^1}{\m{A}_1}$. Note that since $Q_2^1 \subset Q_2^0$, $\chi$ takes no more than $\Delta_{|\m{X}|}$ many values on $\binom{Q_2^1}{\m{A}_0}$. Repeating this procedure $t_{\C}(\m{X})$ times, we end up with $Q_2^{t_{\C}(\m{X})-1}$ such that for every $i~<~t_{\C}(\m{X})$, $\chi$ takes no more than $\tan ^{(2|\m{X}|-1)}(0)$ values on $\binom{Q_2^{t_{\C}(\m{X})-1}}{\m{A}_i}$. Now, in $\Cy$, the substructure supported by $Q_2^{t_{\C}(\m{X})-1}$ is isomorphic to $\Cy$ and since we have \[\binom{Q_2^{t_{\C}(\m{X})-1}}{\m{X}} = \bigcup _{i<t_{\C}(\m{X})} \binom{Q_2^{t_{\C}(\m{X})-1}}{\m{A}_i},\] the map $\chi$ takes no more than $t_{\C} (\m{X})\tan ^{(2|\m{X}|-1)}(0) = T(\m{X})$ values on $\binom{Q_2^{t_{\C}(\m{X})-1}}{\m{X}}$, as required. Thus, $\m{X}$ has a big Ramsey degree $T_{\C}(\m{X})$ in $\C$ and it follows that $T_{\C}(\m{X}) \leq T(\m{X})$.

\subsection{Lower bound for $T_{\C}(\m{X})$: $T(\m{X}) \leq T_{\C}(\m{X})$}

We start with an analogue of Lemma \ref{lem:extprop}. 

\begin{lemma}

\label{claimm:ext}
Every extension of $\m{X}$ in $\p _2$ embeds into every extension of $\Cy$.
\end{lemma}

\begin{proof}
We prove that for every extension $\m{B} = (\Cy , B_1, B_2,<)$ of $\Cy$, the following holds \[ \forall x, y \in \Cy \ \ \forall i\in \{ 1,2\} \ \ \left(x<y \rightarrow \exists z \in B_i \ \ x<z<y\right) \ \ (*)\enspace .\] 

Assuming that $(*)$ holds, $B_1$ and $B_2$ are dense in $\m{B}$. It follows that $\Q _2$, and therefore every element of $\p _2$, embeds into $\m{B}$. In particular, every extension of $\m{X}$ embeds into $\m{B}$, which finishes the proof of Lemma \ref{claimm:ext}. We consequently turn to the proof of $(*)$. Without loss of generality, we may assume that $i=1$. We have several elementary cases to verify:

\begin{enumerate}
\item If $x,y \in B_1$. Fix $z \in \Cy$ such that $x \la{\Cy} z \la{\Cy} y$. Then $z \in B_1$. Indeed, if not, then $z \in B_2$ and so $x > z$ and $z > y$. Hence $x>y$, a contradiction.

\item If $x,y \in B_2$, then $z \in \Cy$ such that $x \ra{\Cy} z \ra{\Cy} y$ works. 

\item If $x \in B_1$ and $y \in B_2$, then $z \in \Cy$ such that $x \la{\Cy} z$ and $y \la{\Cy} z$ works. 

\item If $x \in B_2$ and $y \in B_1$, then $z \in \Cy$ such that $z \la{\Cy} x$ and $z \la{\Cy} y$ works. 
\end{enumerate}

This finishes the proof of Lemma \ref{claimm:ext}.
\end{proof}

We can now show $T(\m{X}) \leq T_{\C}(\m{X})$ by producing a map $\chi$ on $\binom{\Cy}{\m{X}}$ taking $T(\m{X})$ values on the set $\binom{C}{\m{X}}$ whenever $C \in \binom{\Cy}{\Cy}$. First, for every $i<t_{\C}(\m{X})$, Theorem \ref{thm:bigRdP} guarantees the existence of a map $\lambda _i : \funct{\binom{\Q_2}{\m{A}_i}}{[\tan^{(2|\m{X}|-1)}(0)]}$ witnessing that the big Ramsey degree of $\m{A}_i$ in $\p_2$ is equal to $\tan^{(2|\m{X}|-1)}(0)$. Next, consider $\Cy$, seen as $\pr (\Q_2)$. Then given a copy $\mc{X}$ of $\m{X}$ in $\Cy$, the substructure $\m{A}(\mc{X})$ of $\Q_2$ supported by $\mc{X}$ is an extension of $\m{X}$ in $\p _2$ and is isomorphic to a unique element of the family $(\m{A}_i)_{i<t_{\C}(\m{X})}$. Define then the map $\chi : \funct{\binom{\Cy}{\m{X}}}{[t_{\C}(\m{X})]\times[\tan^{(2|\m{X}|-1)}(0)]}$ by \[\chi (\mc{X})=( i , \lambda_i (\m{A}(\mc{X}))\enspace .\] where $i<t_{\C}(\m{X})$ is the unique natural such that $\m{A}(\mc{X}) \cong \m{A}_i$. Then $\chi$ is as required: Let $C \in \binom{\Cy}{\Cy}$. The substructure $\m{B}$ of $\Q_2$ supported by $C$ is an extension of $\Cy$ and by Lemma \ref{claimm:ext}, all the extensions of $\m{X}$ embed in $\m{B}$. Additionnally, $\Q_2$ embeds into $\m{B}$ so $\lambda_i$ takes $\Delta_{|\m{X}|}$ many values on $\binom{\m{B}}{\m{A}_i}$ for every $i$. Thus, $\chi$ takes $T(\m{X})$ many values on $\binom{C}{\m{X}}$. This shows that $T(\m{X}) \leq T_{\C}(\m{X})$ and finishes the proof of Theorem \ref{thm:bigRdC}. 

\subsection{Elementary values of $T_{\C}(\m{X})$}  

As previously, we finish this section by collecting the elementary values of $T_{\C}(\m{X})$:

\begin{figure}[h]
\setlength{\unitlength}{1mm}
\begin{picture}(110,30)(0,0)

\put(2,19){$\m{X}$}
\put(0,5){$T_{\C}(\m{X})$}

\put(15,20){\circle*{1}}
\put(15,5){$2$}

\put(25,20){\circle*{1}}
\put(45,20){\circle*{1}}
\put(44,20){\vector(-1,0){18}}
\put(35,5){$8$}

\put(65,25){\circle*{1}}
\put(55,15){\circle*{1}}
\put(75,15){\circle*{1}}
\put(74,15){\vector(-1,0){18}}
\put(64,24){\vector(-1,-1){8}}
\put(74,16){\vector(-1,1){8}}
\put(65,5){$96$}

\put(95,25){\circle*{1}}
\put(85,15){\circle*{1}}
\put(105,15){\circle*{1}}
\put(86,15){\vector(1,0){18}}
\put(94,24){\vector(-1,-1){8}}
\put(104,16){\vector(-1,1){8}}
\put(95,5){$32$}

\end{picture}
\caption{Elementary values of $T_{\C}(\m{X})$}
\end{figure}
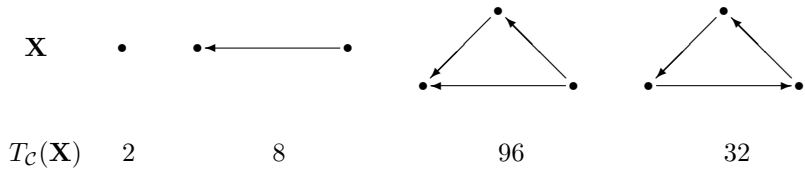

\section{A colored version of Milliken's theorem}

\label{section:Milliken}

In Section \ref{section:bigRdC}, we provided a proof of Theorem \ref{thm:bigRdC} assuming Theorem \ref{thm:bigRdP}. The goal of the present section is to make a first step towards a proof of Theorem \ref{thm:bigRdP} by proving a strengthening of the so-called Milliken theorem. The motivation behind the strategy here really comes from the proof of Theorem \ref{thm:bigRdP} when $n=1$. This was completed by Devlin in \cite{D} thanks to two main ingredients. The first one is a detailed analysis of how copies of $\Q$ may appear inside $\Q$ when $\Q$ is identified with the complete binary tree $[2]^{<\infty}$ of all finite sequences of 0's and 1's ordered lexicographically. The second ingredient is a partition result on trees due to Milliken in \cite{Mi}. In our case, where we are interested in $\Q _n$ instead of $\Q$, the relevant objects to study are not trees anymore but what we will call \emph{colored trees}. In that context, Devlin's ideas can be applied with few modifications to identify how copies of $\Q _n$ may appear inside $\Q_n$. Those are presented in Section \ref{section:bigRdP}. However, the relevant version of Milliken's theorem requires more work and the purpose of the present section is to show it can be completed. We start with a short reminder about the combinatorial structures lying at the heart of Milliken's theorem: order-theoretic trees.  

In what follows, a \emph{tree} is a partially
ordered set $(T, \leq)$ such that given any element $t
\in T$, the set $\{ s \in T : s \leq t \}$ is finite and linearly ordered by $\leq$. The number of predecessors of $t \in T$, $\mathrm{ht}(t) = |\{ s \in T : s < t \}|$ is the \emph{height of} $t
\in T$. The $m$-th level of $T$ is $T(m) = \{t \in T : \mathrm{ht}(t) = m \}$. The \emph{height
of} $T$ is the least $m$ such that $T(m) = \emptyset$ if such an $m$ exists. When no such $m$ exists, we say that $T$ has infinite height. When $|T(0)| = 1$, we say that $T$ is \emph{rooted} and we denote the root of $T$ by $root (T)$. $T$ is \emph{finitely branching} when every element of $T$ has only finitely many immediate successors. When $T$ is a tree, the tree structure on $T$ induces a tree structure on every subset $S \subset T$. $S$ is then called a \emph{subtree} of $T$. Here, all the trees we will consider will be rooted subtrees of the tree $\N^{<\infty}$ of all finite sequences of naturals ordered by initial segment. That is, every element of $\N^{<\infty}$ is a map $t : \funct{[m]}{\N}$
for some natural $m \in \N$. In the sequel, this natural is denoted $|t|$ and is thought as the length of the sequence $t$. The ordering $\leq$ is then defined by $t \leq s$ iff $|t| \leq |s|$ and \[ \forall k \in [|t|], \ \ t(k) = s(k)\enspace .\] 

That is, if we think of $t$ as the sequence of digits $t(0)t(1)\ldots t(|t|-1)$, then $t\leq s$ simply means that $s$ is obtained from $t$ by adding some extra digits to the right of $t$, ie \[ s= t(0)t(1)\ldots t(|t|-1)s(|t|)s(|t|+1)\ldots s(|s|-1)\enspace .\]

The main concept attached to Milliken's theorem is the concept of \emph{strong subtree}. Fix a downwards closed finitely branching subtree $T$ of $\N ^{< \infty}$ with infinite height. Say that a subtree $S$ of $T$ is \emph{strong} when 
\begin{enumerate}
	\item $S$ has a smallest element.  
	\item Every level of $S$ is included in a level of $T$. 
	\item For every $s \in S$ not maximal in $S$ and every immediate successor $t$ of $s$ in $T$ there is exactly one immediate successor of $s$ in $S$ extending $t$. 
\end{enumerate}

An example of strong subtree in provided in Figure \ref{fig:strongsubtree}. For a natural $m>0$, denote by $\mathcal{S}_m (T)$ the set of all strong subtrees of $T$ of height $m$. Denote also by $\mathcal{S}_{\infty} (T)$ the set of all strong subtrees of $T$ of infinite height. 

\begin{figure}[h]
\begin{center}
\hskip-10pt\includegraphics[width=180pt]{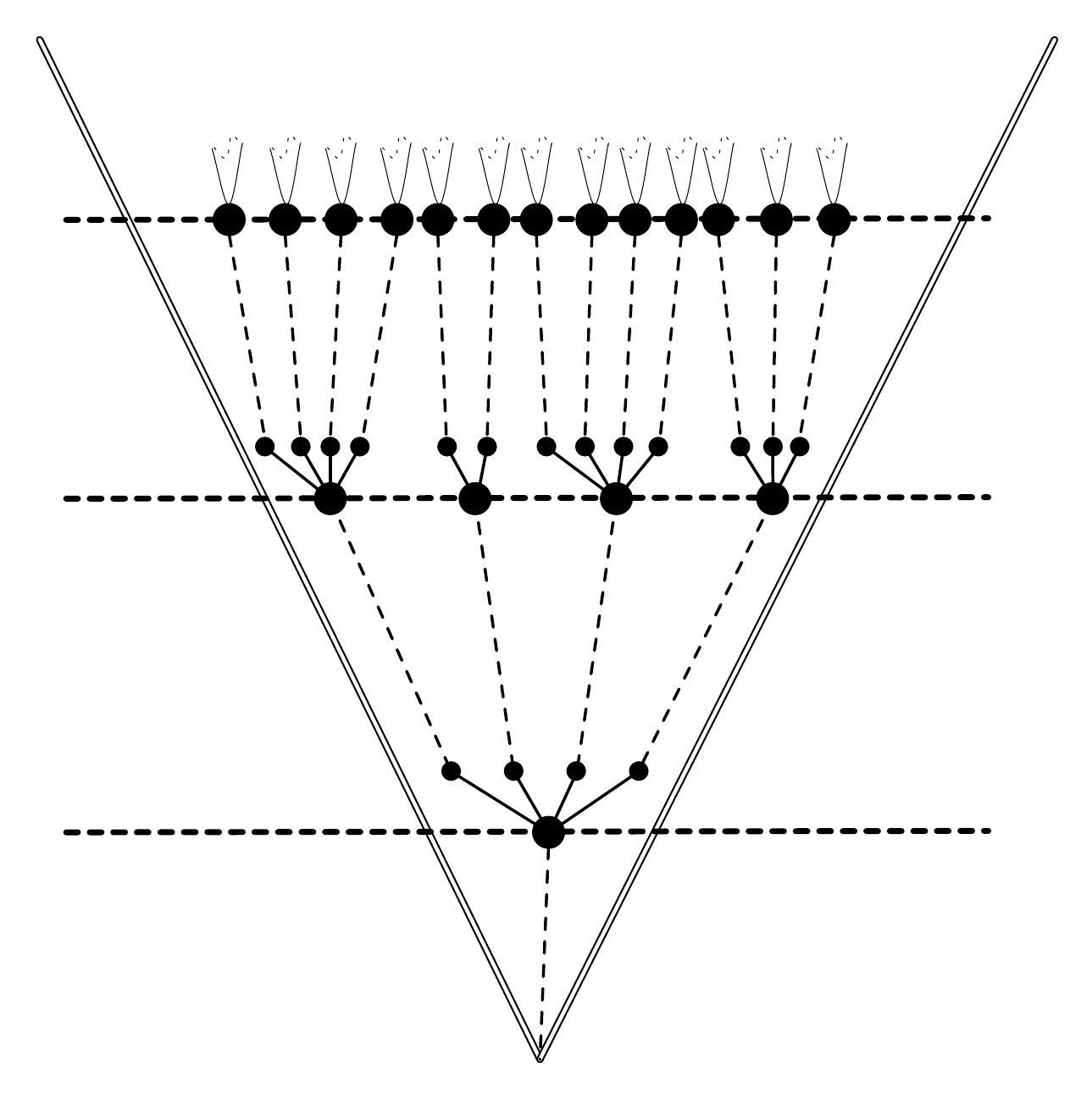}%
\end{center}
\label{fig:strongsubtree}
\caption{A strong subtree}
\end{figure}

\begin{thm}[Milliken \cite{Mi}] 

\label{thm:Milliken}

Let $T$ be a nonempty downward closed finitely branching subtree of $\N ^{< \infty}$ with infinite height. Let $k,m>0$ be naturals. Then for every map $\chi : \funct{\mathcal{S}_m (T)}{[k]}$, there is $S \in \mathcal{S}_{\infty}(T)$ such that $\chi$ is constant on $\mathcal{S}_m (S)$.  

\end{thm}

For our purposes, we need a stronger version of Milliken's theorem relative to \emph{$n$-colored trees}. Let $\alpha \in \N \cup \{ \infty\}$ and $n>0$ be a natural. An $n$-colored tree of height $\alpha$ is a tree $T$ of height $\alpha$ together with an \emph{$n$-coloring sequence} $\tau$ assigning an element of $[n]$ (thought as a color) to each of the levels of $T$ ($\tau(i)$ then corresponds to the color of $T(i)$, the level $i$ of $T$). If $S$ is a strong subtree of $T$, $\tau$ induces an $n$-coloring sequence of $S$ provided by a subsequence of $\tau$. For $\beta \leq \alpha$ and $\sigma$ a sequence of length $\beta$ with values in $[n]$, let $\mathcal{S}_{\sigma}(T)$ denote the set of all strong subtrees of $T$ such that the coloring sequence induced by $\tau$ is equal to $\sigma$. 

\begin{thm} 

\label{thm:Milliken'}

Let $T$ be a nonempty downward closed finitely branching subtree of $\N ^{< \infty}$ with infinite height. Let $n>0$ be a natural and $\Sigma$ an $n$-coloring sequence of $T$ taking each value $i \in [n]$ infinitely many times. Let $k>0$ be a natural and $\sigma$ an $n$-coloring sequence with finite length. Then for every map $\chi : \funct{\mathcal{S}_{\sigma} (T)}{[k]}$, there is $S \in \mathcal{S}_{\Sigma}(T)$ such that $\chi$ is constant on $\mathcal{S}_{\sigma}(S)$.  

\end{thm}

\begin{proof}
We proceed by induction on $n$. The case $n=1$ is handled by the original version of Milliken's theorem. We therefore concentrate on the induction step. Assume that Theorem \ref{thm:Milliken'} holds for the natural $n$. We show that it also holds for the natural $n+1$. Since $\Sigma$ takes each value $i \in [n]$ infinitely many times, then by going to a subtree of $T$ if necessary, we may arrange that $\Sigma$ is the sequence defined by \[ \Sigma (k) = k \mod(n+1)\enspace .\] 

We may also assume that for every $t \in T$, the set $\{ j \in \omega : t ^{\frown}j \in T\}$ is an initial segment of $\N$ (here, $t ^{\frown}j$ denotes the concatenation of $t$ and $j$, that is the sequence obtained from $t$ by extending it with the extra digit $j$. Formally $t ^{\frown}j (n) = t(n)$ for every $n<|t|$ and $t ^{\frown}j (|t|) = j$). 

Let $\pj : \funct{T}{T}$ be the function mapping the elements of $T$ with color $n$ onto their immediate predecessor in $T$ and leaving the other elements of $T$ fixed. For an $(n+1)$-coloring sequence $\tau$, let $\pj(\tau)$ be the $n$-coloring sequence obtained from $\tau$ by replacing every occurence of $n$ in $\tau$ by $(n-1)$.

Say that a strong subtree $U$ of $T$ satisfies $(*)$ when
\begin{enumerate}
	\item For every $u \in U$ and for every immediate succesor $u'$ of $u$ in $U$, if $u$ has color $(n-1)$ and $t$ is the immediate successor of $u$ in $T$ such that $t\leq u'$, then $t^{\frown}0 \leq u'$. 
	\item For every $u \in U$ with color $n$, $u = \pj(u)^{\frown}0$.  
\end{enumerate}

\begin{lemma}
\label{lem:sigmaT}
Let $\sigma$ be an $(n+1)$-coloring sequence with finite length, $S \in \mathcal{S}_{\pj (\sigma)}(\pj (T))$. Then there is a unique $\sigma ^* T$ in $\mathcal{S}_{\sigma}(T)$ such that 
\begin{itemize}
	\item $\sigma ^* T$ satisfies $(*)$.
	\item For every $k$, $\pj (\sigma ^* T(k)) \subset S(k)$. 
\end{itemize}
\end{lemma}

Assuming Lemma \ref{lem:sigmaT}, the induction step can be carried out as follows: let $\sigma$ be an $(n+1)$-coloring sequence with finite length and $\chi : \funct{\mathcal{S}_{\sigma}(T)}{[k]}$. Using Lemma \ref{lem:sigmaT}, transfer $\chi$ to $\lambda : \funct{\mathcal{S}_{\pj (\sigma)}(\pj(T))}{[k]}$ by setting $\lambda(S) = \chi (\sigma^*S)$. Then, using Theorem \ref{thm:Milliken'} for the natural $n$, find a strong subtree $U$ of $\pj (T)$ with coloring sequence $\pj (\Sigma)$ such that $\mathcal{S}_{\pj (\sigma)}(U)$ is $\lambda$-monochromatic with color $\varepsilon$. By refining $U$ if necessary, we may assume that no two consecutive levels of $U$ are consecutive in $T$. Then $\Sigma^*U \in \mathcal{S}_{\Sigma}(T)$ and satisfies $(*)$. We claim that $\chi$ is constant on  $\mathcal{S}_{\sigma}(\Sigma^*U)$. Indeed, let $V \in \mathcal{S}_{\sigma}(\Sigma^*U)$. Then $\pj (V) \subset \pj (\Sigma^*U)\subset U$ and it has coloring sequence $\pj(\sigma)$. Let $W \subset U$ be a strong subtree with the same height as $\pj (V)$ and such that $\pj (V) \subset W$. Since $\Sigma^*U$ has property $(*)$, so does $V$. By Lemma \ref{lem:sigmaT}, it follows that $V = \sigma^*W$. Hence \[\chi(V) = \chi (\sigma^*W) = \lambda(W) = \varepsilon. \qedhere\]\end{proof}

\begin{proof}[Proof of Lemma \ref{lem:sigmaT}]
For a tree $V$ and an element $v \in V$, let $IS_V(v)$ denote the set of all immediate successors of $v$ in $V$. We start by proving the existence of a tree $U$ fulfilling the requirements. We proceed inductively and construct $U$ level by level. For $U(0)$, we distinguish two cases. If $\sigma(0)\neq n $, we set $U(0) = S(0) (= \{root(S)\})$. If $\sigma(0) = n $, we set $U(0) = \{root(S)^{\frown}0\}$. Assume that $U(0)\ldots U(k)$ are constructed. 

Case 1: $\sigma(k) \neq n-1$. 
Then for every $u \in U(k)$, any element $v$ of $IS_T(u)$ is also in $IS_{\pj (T)}(\pj(u))$. Thus, there is a unique $\phi (v) \in S(k+1)$ such that $v \leq \phi (v)$. If $\sigma (k+1)\neq n $, $U(k+1)$ is formed by collecting all the $\phi(v)$'s. Otherwise, $\sigma (k+1) = n$ and $U(k+1)$ is formed by collecting all the $\phi (v)^{\frown}0$'s.

Case 2: $\sigma(k) = n-1$.

Then the immediate successors of the elements of $U(k)$ in $T$ have color $n$ and are not in $\pj (T)$. For $u \in U(k)$ and $v \in IS_T(u)$, $v \notin IS_{\pj (T)}(\pj(u))$ and $v$ may be dominated by more than one element in $S$. However, $v^{\frown}0 \in IS_{\pj (T)}(\pj(u))$ is dominated by exactly one element in $S$. Let $\phi(v)$ denote this element. Form $U(k+1)$ as in Case 1 by collecting all the $\phi (v)$'s if $\sigma (k+1) \neq n $ and all the $\phi (v)^{\frown}0$'s otherwise. 

Repeating this procedure, we end up with a tree $U$. This tree is as required as at every step, the construction makes sure that it is strong and that the property $(*)$ is satisfied.

We now show that this procedure is actually the only possible one. Assume that $U$ and $U'$ are as required. We show that $U = U'$. First of all, it should be clear that $U$ and $U'$ have the same root. We now show that if $u \in U \cap U'$, then $IS_U(u) = IS_{U'}(u)$. It suffices to show that $IS_U(u) \subset IS_{U'}(u)$. Let $w \in IS_U(u)$. 

\begin{claimm}
$\pj (w) \in IS_{\pj(U)}(\pj(u)) $.
\end{claimm} 
  
\begin{proof}
Let $v \in U$ be such that $\pj(v) \leq \pj(w)$ and $u<v$. Since $\pj(v)$ and $\pj (w)$ are comparable, $v$ and $w$ are above the same immediate successor of $u$ in $U$. Hence $w \leq v$ and $\pj(w) \leq \pj(v)$. \end{proof}

So, fix $t \in IS_T(u)$ and $v \in IS_{\pj(T)}(\pj(u))$ such that \[ u\leq t \leq w \ \ \mathrm{and} \ \ \pj(u) \leq v \leq \pj(w)\enspace .\] 

Observe that because $\pj (T) \subset T$, we have $t\leq v$. Let $w' \in IS_{U'}(u)$ be such that $u \leq t \leq w'$. Note that as for $w$, we have $\pj (w') \in IS_{\pj( U')}(\pj(u))$. 

\begin{claimm}
$v \leq \pj(w')$.
\end{claimm}

\begin{proof}
If $t \in \pj(T)$, then $t=v$ and we are done. Otherwise, $t$ has color $n$ and $u \leq t < v \leq w$. By $(*)$ for $U$, $t^{\frown}0 \leq w$. Hence $t^{\frown}0 = v$. Now, by $(*)$ for $U'$, we have $t^{\frown}0 \leq w'$. Hence, $v \leq w'$ and $v \leq \pj (w')$. \end{proof}

It follows that $\pj (w)$ and $\pj (w')$ are in $S$ and above $v$. Since they have the same height, they must be equal. Hence, $w=w'$. \end{proof}

\section{Big Ramsey degrees in $\p _n$}

\label{section:bigRdP}

In this section, we show how Theorem \ref{thm:bigRdP} can be proven thanks to the machinery developed in Section \ref{section:Milliken}. As already mentioned, this is essentially done by using the ideas that were used by Devlin in \cite{D} to study the partition calculus of the rationals. For that reason, several results are stated without proof. Our presentation here, however, follows a different path. Namely, it repeats the exposition of the forthcoming book \cite{T}. All the details of the proofs that we omit here will appear in \cite{T} together with a wealth of other applications of Milliken's theorem.     

In the sequel, we work with the tree $T = [2]^{<\infty}$ of finite sequences of $0$'s and $1$'s colored by the map $\Sigma$ defined by $\Sigma(i) = (i \mod n) +1$ for every $i \in \N$. Noticing that $(T,<_{lex})$ and $(\Q,<)$ are isomorphic linear orderings and that in $(T,<_{lex})$, the subset $T_i$ of all the elements with color $i$ is dense whenever $i=1\ldots n$, we see that the colored tree $T$ is isomorphic to $\Q _n$. 

For $s , t \in T$, set \[ s \wedge t = \max \{ u \in T : u \subset s, \ u \subset t \}\enspace .\]

For $A \subset T$, set \[ A^{\wedge} = \{ s \wedge t : s, t \in A\}\enspace .\]

Note that $A \subset A^{\wedge}$ and that $A^{\wedge}$ is the minimal rooted subtree of $T$ containing $A$. Define an equivalence relation $\mathrm{Em}$ on the collection of all finite subsets of $T$ as follows: for $A, B \subset T$, set $A \mathrm{Em} B$\index{$A \mathrm{Em} B$} when there is a bijection $f : \funct{A^{\wedge}}{B^{\wedge}}$ such that for every $s, t \in A^{\wedge}$:

\vspace{0.5em}
\hspace{1em} 
i) $s \leq t \leftrightarrow f(s) \leq f(t)$.

\vspace{0.5em}
\hspace{1em} 
ii) $|s|<|t| \leftrightarrow |f(s)|<|f(t)|$. 

\vspace{0.5em}
\hspace{1em} 
iii) $s \in A \leftrightarrow f(s) \in B$.

\vspace{0.5em}
\hspace{1em} 
iv) $t(|s|)=f(t)(|f(s)|)$ whenever $|s|<|t|$. 

\vspace{0.5em}
\hspace{1em} 
iv) $f(s)$ has color $i$ whenever $s$ has color $i$.

\vspace{0.5em}

It should be clear that $\mathrm{Em}$\index{$\mathrm{Em}$} is an equivalence relation. Given $A \subset T$, let $[A]_{\mathrm{Em}}$ denote the $\mathrm{Em}$-equivalence class of $A$. Let also $\sigma_A$ denote the sequence of colors corresponding to $A^{\wedge}$. 

\begin{lemma}
\label{lem:envelope}
Let $V \in \mathcal{S}_{\sigma}(T)$ and $A \subset T$ such that $\sigma_A = \sigma$. Then there is a unique $A' \in [A]_{\mathrm{Em}}$ such that $A' \subset V$.  
\end{lemma}

Finally, for a strong subtree $S$ of $T$, let $\restrict{[A]_{\mathrm{Em}}}{S}$ denote the set of all elements of $[A]_{\mathrm{Em}}$ included in $S$. 

\begin{thm}
Let $A$ be a finite subset of $T$. Then for every natural $k>0$ and every map $\chi : \funct{[A]_{\mathrm{Em}}}{[k]}$, there is $S \in \mathcal{S}_{\Sigma}(T)$ such that $\chi$ is constant on $\restrict{[A]_{\mathrm{Em}}}{S}$.  
\end{thm}

\begin{proof}
Define $\lambda (V)$ for every $V \in \mathcal{S}_{\sigma _A}(T)$ by the $\chi$-value of its unique subset which
belongs to $[A]_{\mathrm{Em}}$. According to Lemma \ref{lem:envelope}, the map $\lambda$ is well-defined. By Theorem~\ref{thm:Milliken'}, there is $S \in \mathcal{S}_{\Sigma}(T)$ such that $\lambda$ is constant on $\mathcal{S}_{\sigma}(S)$. It follows that $\chi$ is constant on $\restrict{[A]_{\mathrm{Em}}}{S}$. \end{proof}

As a direct consequence, every element $\m{X}$ of $\p_n$ has a big Ramsey degree in $\p _n$ less or equal to the number of embedding types of $\m{X}$ inside $T$. It turns out that when reconstituting copies of $\Q _n$ inside $T$, certain embedding types can be avoided. 

A finite set $A \subset T$ realizes a \emph{Devlin embedding type} when
\begin{enumerate}
	\item $A$ is the set of all terminal nodes of $A^{\wedge}$. 
	\item $|s|\neq|t|$ whenever $s \neq t \in A^{\wedge}$.
	\item $t(|s|) = 0$ for all $s, t \in A^{\wedge}$ such that $|s| < |t|$ and $s \nleq t$.	
\end{enumerate}

Figure \ref{fig:Devlintypes2} represents eight of the sixteen Devlin types that may be realized by a 3-element subset of $T$ in the uncolored case ($n=1$) (Each picture represents a subset of the binary tree).   

\begin{figure}[h]
\begin{center}
\hskip-10pt\includegraphics[scale=0.8]{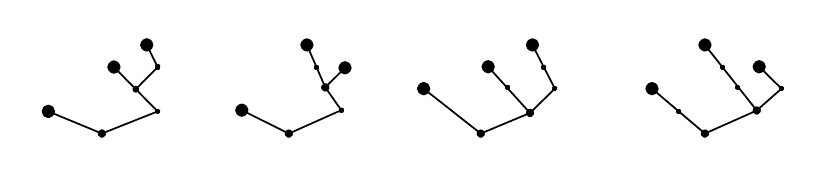}%
\drawat{-59mm}{0mm}{$\emptyset$}
\drawat{-59mm}{12mm}{$1010$}
\drawat{-66mm}{8mm}{$100$}
\drawat{-69mm}{4.5mm}{$0$}%
\hskip-10pt\includegraphics[scale=0.8]{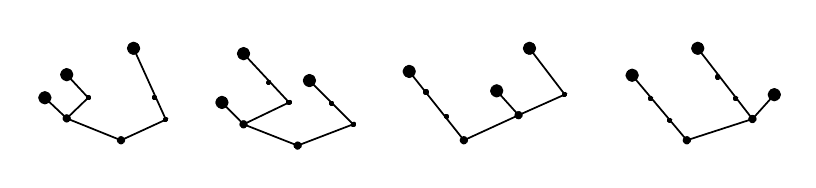}%
\end{center}
\caption{Examples of Devlin types on three elements when $n=1$}
\label{fig:Devlintypes2}
\end{figure}

\begin{lemma}
Every $S \in \mathcal{S}_{\Sigma}(T)$ includes an antichain $X$ such that:
\begin{enumerate}
	\item $(X,X\cap T_1,\ldots,X\cap T_n,<_{lex})$ is isomorphic to $\Q_n$,
	\item Every finite subset of $X$ realizes a Devlin embedding type. 
	\item For every Devlin embedding type $[A]_{\mathrm{Em}}$ and every $Y \subset X$ isomorphic
to $\Q_n$ there exists $B \subset Y$ such that $[B]_{\mathrm{Em}} = [A]_{\mathrm{Em}}$.
\end{enumerate}
\end{lemma}

\begin{proof}
Without loss of generality, we may assume that $S=T$. Let $W\subset T$ be the $\wedge$-closed subtree of $T$ uniquely determined by the following properties (For an attempt to represent the lowest levels of $W$, see Figure \ref{fig:Devlin}): 
\begin{enumerate}
	\item $root (W) = \emptyset$.
	\item $\forall l\in \N \ \forall 0<i<n \ \ |W\cap T(nl)|=1 \ \ \mathrm{and}  \ \ |W\cap T(nl+i)|=0$.
	\item $\forall l\in \N \ \forall s,t \in W(l) \ \ s<_{lex}t \rightarrow |s|<|t|$.
	\item $\forall l<m\in\N \ \forall s \in W(l) \ \forall t\in W(m) \ \ |s|<|t|$.
	\item $W$ is order-isomorphic to $(T,<_{lex})$. 
	\item $\forall s \in W \ \forall t<s \ \ t \notin W \rightarrow t^{\frown}0 < s$. 
\end{enumerate} 

\begin{figure}[h]
\begin{center}
\hskip-10pt\includegraphics[scale=0.8]{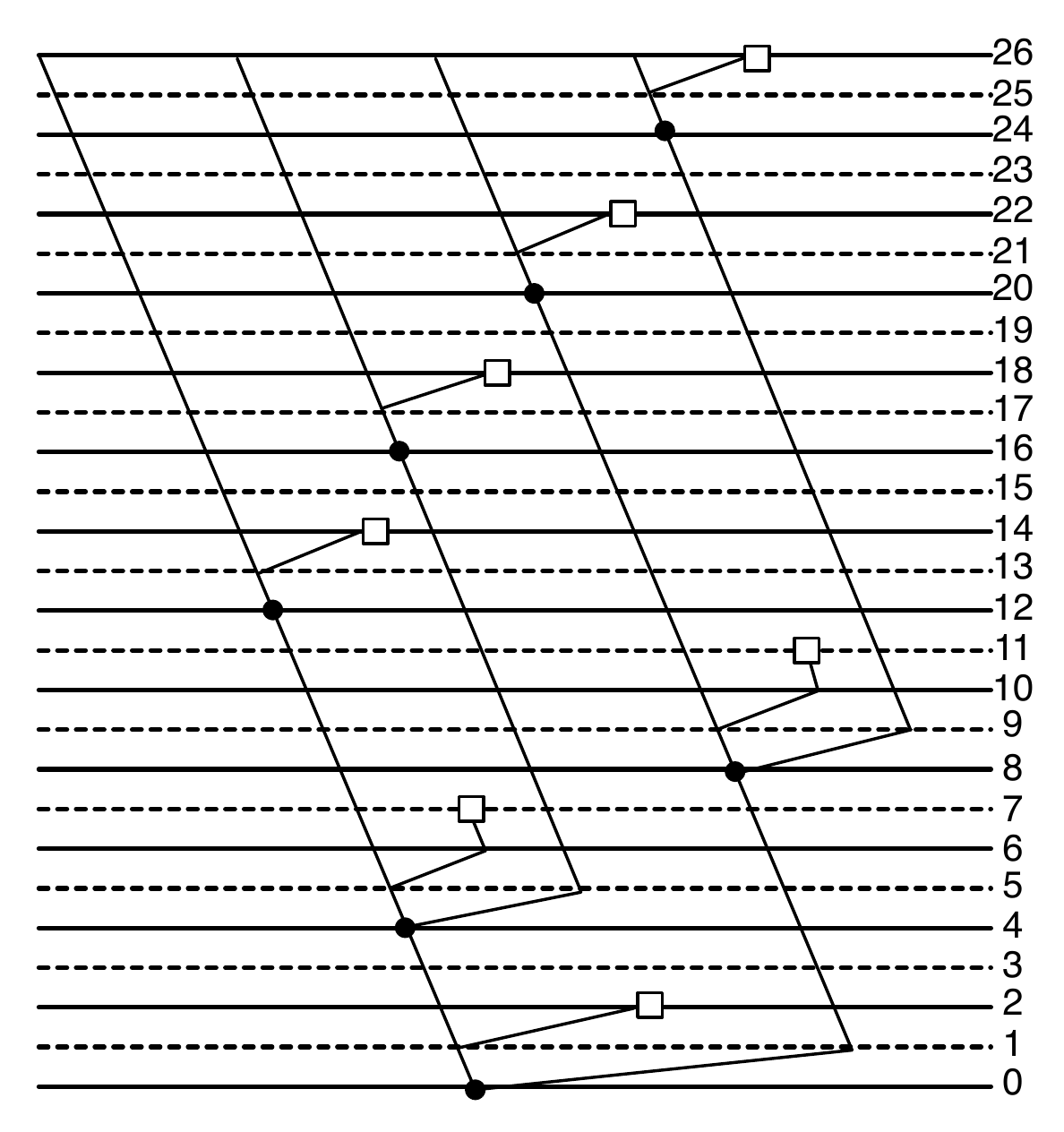}%
\drawat{-35.5mm}{12.5mm}{$x_\emptyset$}%
\drawat{-65.49mm}{19.32mm}{$w_0$}%
\drawat{-52.18mm}{30.53mm}{$x_0$}%
\drawat{-34.5mm}{34mm}{$w_1$}%
\drawat{-21.7mm}{44.5mm}{$x_1$}%
\drawat{-78.45mm}{48.35mm}{$w_{00}$}%
\drawat{-60.76mm}{55.65mm}{$x_{00}$}%
\drawat{-67.5mm}{63mm}{$w_{01}$}%
\drawat{-50.5mm}{70mm}{$x_{01}$}%
\drawat{-54.57mm}{77.15mm}{$w_{10}$}%
\drawat{-38.68mm}{84.4mm}{$x_{10}$}%
\drawat{-42.44mm}{91.73mm}{$w_{11}$}%
\drawat{-26mm}{99mm}{$x_{11}$}%
\end{center}
\caption{$W=\{w_f: f\in T\}$ and $X=\{x_f: f\in T\}$ when $n=2$}
\label{fig:Devlin}
\end{figure}

Let $f \mapsto w_f$ denote the isomorphism between $(T,<_{lex})$ and $W$. Define then $x_f = w_f^{\frown}01^{\frown}0^i$ (here, $0^i$ denotes the sequence with $i$ many $0$'s) where $i$ is such that $0\leq i <n$ and $|f| = i \mod(n)$. Then one can check that for every $Y \subset X$ isomorphic to $\Q_n$, the embedding types of the finite subsets of $Y$ are exactly the Devlin's embedding types. \end{proof}

It follows that every element $\m{X}$ of $\p_n$ has a big Ramsey degree in $\p _n$ equal to the number of embedding types of $\m{X}$ inside $T$. Proceeding by induction on the size of $\m{X}$, it can be shown that this number of embeddings actually only depends on the size of $\m{X}$ and satisfies a recursion formula which allows to identify it with the number $\tan ^{(2|\m{X}|-1)}(0)$. This finishes the proof of Theorem \ref{thm:bigRdP}.


\begin{thebibliography}{KPT}

\bibitem[Aus88]{A}
J. Auslander, Minimal flows and their extensions, North Holland, 1988. 

\bibitem[Ch98]{Ch}
G. L. Cherlin, The classification of countable homogeneous directed graphs and countable
homogeneous n-tournaments, \emph{Mem. Amer. Math. Soc. 131}, 621, xiv+161, 1998.

\bibitem[Dev79]{D}
D. Devlin, Some partition theorems and ultrafilters on $\omega$, Ph.D. Thesis, Dartmouth College, 1979. 

\bibitem[Hen72]{H}
C. W. Henson, Countable homogeneous relational structures and $\aleph _0$-categorical theories, \emph{J. Symbolic Logic}, 37, 494-500, 1972. 

\bibitem[KPT05]{KPT}
A. S. Kechris, V. Pestov and S. Todorcevic, Fra\"iss\'e limits,
Ramsey theory, and topological dynamics of automorphism groups,
\emph{Geom. Funct. Anal.}, 15, 106-189, 2005. 

\bibitem[La84]{L}
A. H. Lachlan, Countable homogeneous tournaments, \emph{Trans. Amer. Math. Soc.}, 284 (2), 431-461, 1984. 

\bibitem[LSV04]{LSV}
C. Laflamme, N. W. Sauer V. Vuksanovic, Canonical partitions of universal structures, \emph{preprint}, 2004.

\bibitem[Mi79]{Mi}
K. Milliken, A Ramsey theorem for trees, \emph{J. Comb. Theory}, 26, 215-237, 1979. 

\bibitem[Pe98]{Pe}
V. Pestov, On free actions, minimal flows, and a problem by Ellis, \emph{Trans. Amer. Math. Soc.}, 350, 4149-4165, 1998. 

\bibitem[To]{T}
S. Todorcevic, Introduction to Ramsey spaces, \emph{to appear}.

\bibitem[W76]{W}
R. E. Woodrow, Theories with a finite set of countable models and a small language, Ph.D. Thesis, Simon Fraser University, 1976. 

\end{thebibliography}
\end{document}